\documentclass[11pt]{amsproc}
\usepackage[margin=1in]{geometry}
\usepackage{setspace,fullpage}
\usepackage{graphicx}
\usepackage[nice]{nicefrac}
\usepackage{amssymb}
\usepackage{multirow}
\usepackage{array}
\usepackage{hyperref}

\usepackage{amsmath,amsthm,amscd,amssymb, mathrsfs}
\usepackage{mathtools}
\usepackage{latexsym}

\numberwithin{equation}{section}

\theoremstyle{plain}
\newtheorem{theorem}{Theorem}[section]
\newtheorem{lemma}[theorem]{Lemma}
\newtheorem{corollary}[theorem]{Corollary}
\newtheorem{proposition}[theorem]{Proposition}
\newtheorem{conjecture}[theorem]{Conjecture}

\newtheorem{algorithm }[theorem]{Algorithm}

\theoremstyle{definition}

\theoremstyle{remark}
\newtheorem{remark}[theorem]{Remark}

\newtheorem{case[theorem]}{Case}

\theoremstyle{plain}
\newtheorem{definition}[theorem]{Definition}

\title[\parbox{14cm}{\centering{ Line Set Distance }} \quad]{The Erd\H{o}s-Falconer distance problem between  arbitrary sets  and $k$-coordinatable sets in finite fields}
\author{Hunseok Kang, Doowon Koh, and Firdavs Rakhmonov}
 
\address{College of Engineering and Technology\\
American University of the Middle East\\
Kuwait}
\email{hunseok.kang@aum.edu.kw}

\address{Department of Mathematics\\
Chungbuk National University \\
Cheongju, Chungbuk 28644 Korea}
\email{koh131@chungbuk.ac.kr}

\address{School of Mathematics and Statistics\\
University of St Andrews \\
Scotland}
\email{fr52@st-andrews.ac.uk}

\thanks{
Doowon Koh was supported by the National Research Foundation of Korea
(NRF) grant funded by the Korea government (MSIT) (NO. RS-2023-00249597).\\
Firdavs Rakhmonov was  financially supported by a  \emph{Leverhulme Trust Research Project Grant} (RPG-2023-281).}

\begin{document} 

\begin{abstract} 
In this paper, we study the cardinality of the distance set $\Delta(A, B)$ determined by two subsets $A$ and $B$ of the $d$-dimensional vector space over a finite field $\mathbb{F}_q$.   Assuming that $A$ or $B$ lies in a $k$-coordinate plane up to translations and rotations, we prove that if $|A||B| > 2q^d$, then $|\Delta(A, B)| > q/2$, where $|\Delta(A, B)|$ denotes the number of distinct distances between elements of $A$ and $B$. 
In particular, we show that our result recovers the sharp $(d+1)/2$ threshold for the Erdős–Falconer distance problem in odd dimensions, where distances are determined by a single set. As an application, we also obtain an improved result on the Box distance problem posed by Borges, Iosevich, and Ou, in the case where $2$ is a square in $\mathbb{F}_q$.
\end{abstract}

\maketitle
\section{Introduction}
Let $\mathbb F_q^d$, $d\ge 1$, be the $d$-dimensional vector space over the finite field $\mathbb F_q$ with $q$ elements, where $q$ is always assumed to be an odd prime power.
Given a set $A\subseteq \mathbb F_q^d$, $d\ge 2$, the distance set, denoted by $\Delta(A)$, is defined as
$$\Delta (A) \coloneqq \{ \| x-y \| : x, y \in A\},$$
where $\|\alpha\|\coloneqq\sum\limits_{i=1}^d \alpha_i^2$ for $\alpha=(\alpha_1, \ldots, \alpha_d)\in \mathbb F_q^d$.

The Erd\H{o}s-Falconer distance problem over general finite fields asks us to determine the smallest exponent $s>0$ such that if $A\subseteq \mathbb F_q^d$ with $|A|\ge C q^s$ for a sufficiently large constant independent of $q$, then $|\Delta(E)|\gtrsim q$.
Here, and throughout the paper, $X \gtrsim Y$ means that there exists a constant $C$ independent of the parameter $q$ so that $X\ge C Y$. We also write $X\lesssim Y$ to mean $Y\gtrsim X$, and we denote $X\sim Y$ when both $X\gtrsim Y$ and $X\lesssim Y$ hold. This problem was introduced by Iosevich and Rudnev \cite{IR07} in 2007 and is considered an extension of the Erdős distance problem over prime fields to more general fields. The original problem was initially studied by Bourgain, Katz, and Tao \cite{BKT}.

Iosevich and Rudnev \cite{IR07} obtained the $s={(d+1)}/{2}$ result on the Erd\H{o}s-Falconer distance problem for all dimensions $d \ge 2$ by relating it to estimates of the Kloosterman sum. In addition, they conjectured that this result could be improved to $s = {d}/{2}$. However, within a few years, counterexamples were found in \cite{HIKR}, showing that this conjecture is not true for general odd dimensions $d \geq 3 $. More precisely, it was shown that the result $s = {(d+1)}/{2}$, proved by Iosevich and Rudnev \cite{IR07}, is optimal for odd dimensions $d \ge 3$. However, the construction of their counterexamples applied only to the case when $ d \equiv 3 \pmod{4}$ or $q \equiv 3 \pmod{4}$. In this paper, we will prove that the result ${(d+1)}/{2}$ due to Iosevich and Rudnev \cite{IR07} is optimal even in the remaining odd dimensional cases with any restricted conditions on the finite field $\mathbb F_q$.  Additionally, for the case when the dimension $d$ is even, we will provide a proof different from the one given in \cite{HIKR}, showing that no better result than ${d}/{2}$ can be obtained in even dimensions $d\ge 2$.

\begin{proposition} \label{P1.1} Let $d\ge 2$ be an integer.
\begin{itemize} \item[(i)] If $d$ is odd, then  for each $\delta>0$,
there exists a set $A \subseteq \mathbb F_q^d$ such that $|A|\sim q^{\frac{d+1}{2}-\delta}$ and it fails to satisfy that $|\Delta(A)|\sim q$.

\item[(ii)] If $d$ is even, then for each $\delta>0$, there exists a set $A\subseteq \mathbb F_q^d$ such that $|A|\sim q^{\frac{d}{2}-\delta}$ and it fails to satisfy that $|\Delta(A)|\sim q$.
\end{itemize}
\end{proposition}

Although the ${(d+1)}/{2}$ result is known as the answer to the Erd\H{o}s-Falconer problem in odd dimensions, it is conjectured that in even dimensions, the result can be lowered to ${d}/{2}$. As a result in this direction, for $d=2$, an improved bound of ${4}/{3}$ has been established, surpassing the optimal ${(d+1)}/{2}$ result in odd dimensions (see, for example, \cite{BHIPJR, CEHIK10, HLR}). In particular, when restricting to the case $d=2$ and $\mathbb F_q$ being a prime field, a further improvement to ${5}/{4}$ has recently been proven by Murphy, Petridis, Pham, Rudnev, and Stevens \cite{MPPRS}. However, in all even dimensions, it remains a challenging open question to prove the following conjecture.

\begin{conjecture}
Suppose that $d\ge 2$ is an even integer and $A\subseteq \mathbb F_q^d$.
If $|A|\ge C q^{\frac{d}{2}}$ for a sufficiently large constant $C$ independent of $q$, then
$|\Delta(A)|\sim q$.
\end{conjecture}

In an effort to understand the Erd\H{o}s-Falconer conjecture, generalized versions and various modified forms of this problem have been introduced and studied.
As the most natural generalization of the Erd\H{o}s-Falconer distance problem, one considers the generalized distance set 
$$
\Delta(A, B) \coloneqq \{\|x - y\| : x \in A, y \in B\}
$$
determined by two sets \(A\), \(B \subseteq \mathbb{F}_q^d\). Then, one can ask the question of determining the minimum size of \(|A||B|\) required to ensure that \(|\Delta(A, B)| \sim q\). In \cite{S06}, Shparlinski proved that if $A, B \subseteq \mathbb F_q^d$, then 
$$ |\Delta(A, B)| \ge \frac{1}{2} \min\left\{ q, ~ \frac{|A||B|}{q^d} \right\}.$$
This clearly implies that if $A, B\subseteq \mathbb F_q^d$ with $|A||B|\ge q^{d+1}$, then $|\Delta(A, B)|\ge q/2$. The exponent $(d+1)$ is optimal for odd dimensions $d\ge 3$, which follows from Proposition \ref{P1.1} (i). On the other hand, it was proven by Koh and Shen \cite{KS} that the exponent \( d+1 \), which is the sharp exponent in odd dimensions $d$, can be lowered to ${8}/{5}$ in the case of \( d = 2 \). However, for higher even dimensions $d\ge 4$, the exponent $(d+1)$ is still the best known result on the generalized Erd\H{o}s-Falconer distance problem, and the conjectured exponent is $d$ as stated below.

\begin{conjecture} \label{con1.3}
Let $d\ge 2$ be an even integer. If $A, B\subseteq \mathbb F_q^d$ with $|A||B|\ge C q^{d}$ for a sufficiently large constant $C$ independent of $q$, we have
$|\Delta(A, B)|\sim q$.
\end{conjecture} 
One may be interested in finding properties of sets $A$ and $B$ that satisfy Conjecture \ref{con1.3}. In the paper \cite{KPV}, conditions on the dimension \( d \) and the finite field \( \mathbb F_q \) were introduced under which Conjecture \ref{con1.3} holds when the set \( A \) is an arbitrary subset of \( \mathbb F_q^d \) and the set \( B \) lies on a sphere or a paraboloid in $\mathbb F_q^d$. 

The main goal of this paper is to provide conditions under which Conjecture \ref{con1.3} holds when $B$ lies in an affine $k$-plane and to apply this result to the Box distance problem studied by Borges, Iosevich, and Ou \cite{BIO}, yielding an improved result.

\subsection{Main results}
We begin by introducing some definitions to clearly state our main results. In $\mathbb F_q^d$, a $k$-dimensional coordinate plane is a $k$-dimensional subspace that includes the origin and extends along $k$-coordinate axes. 

\begin{definition} [$k$-dimensional coordinate plane] In $\mathbb F_q^d$, for any set of $k$ coordinate axes $\{i_i, i_2, \ldots, i_k\} \subseteq \{1, 2, \ldots, d\}$, the $k$-dimensional coordinate plane is defined as the subspace where only the coordinates corresponding to these axes vary freely, while the remaining coordinates are fixed at zero. More formally, a $k$-dimensional coordinate plane is the set
$$P_{i_1, i_2, \dots, i_k} \coloneqq \{ (x_1, x_2, \dots, x_d) \in \mathbb F_q^d \mid x_j = 0 \text{ for all } j \notin \{ i_1, i_2, \dots, i_k \}\}.$$
\end{definition}

\begin{definition} We say that a $k$-dimensional affine subspace $\Pi_k \subseteq \mathbb F_q^d$ is a $k$-coordinatable plane if $\Pi_k$ can be transformed to a $k$-dimensional coordinate plane by a rotation or a translation.  A subset of a $k$-coordinatable plane is called as a $k$-coordinatable set.
\end{definition}
In $\mathbb F_q^d$, not every $k$-dimensional affine subspace is a $k$-coordinatable plane. For example, one easily checks that the line $y=x$ in $\mathbb F_q^2$ is not a $1$-coordinatable plane unless $2\in \mathbb F_q$ is a square number.

Our main result is stated as follows.
\begin{theorem}\label{mainthm}
Suppose that $A\subseteq \mathbb F_q^d$ or $B\subseteq \mathbb F_q^d$ is a $k$-coordinatable set. If $|A||B|\ge 2q^d$, then $|\Delta(A, B)|\ge q/2$.
\end{theorem}
The two-dimensional version of Theorem \ref{mainthm} can be stated more specifically, since in two dimensions, a 1-coordinate plane is one of the lines obtained by translating the line \( x_2 = \lambda x_1 \), where \( x_2 = \lambda x_1\) can be transformed into the $x_1$-axis by rotation. To this end, we introduce some notations.
\begin{definition} \label{DefTL}
The line \( x_2 = \lambda x_1 \) translated by a vector \( (a, b) \in \mathbb{F}_q^2 \) is denoted as \( L_\lambda(a, b) \). That is,  
\[
L_\lambda(a, b) = \{(x_1 + a, x_2 + b) \in \mathbb{F}_q^2 \mid x_2 = \lambda x_1\}.
\]  
In particular, when \( (a, b) = (0,0) \), we denote \( L_\lambda = L_\lambda(a, b) \).
\end{definition}

We denote by $\eta$ the quadratic character of $\mathbb F_q^*$, which is defined by
\[\eta(t) =
\begin{cases}
1, & \text{if } t \text{ is a square in } \mathbb{F}_q^* , \\
-1, & \text{otherwise}.
\end{cases}\]

In Lemma \ref{1Rotation}, we will observe that if $1+\lambda^2$ is a square, then the line $L_\lambda(a,b)$ is a $1$-coordinatable plane in $\mathbb F_q^2$ for all $(a, b)\in \mathbb F_q^2$. From this observation, we are able to obtain a more concrete two-dimensional version of Theorem \ref{mainthm} as follows:

\begin{corollary} \label{mainthmC}  Let $(a, b)\in \mathbb F_q^2$ and let $\lambda\in \mathbb F_q^*$ with $\eta(1+\lambda^2)=1$. If $|A||B|\ge 2 q^2$ for $A\subseteq \mathbb F_q^2$ and $B\subseteq L_\lambda(a,b)$, then $|\Delta(A, B)|\ge q/2$.
\end{corollary}

\subsection{Results from the application of Theorem \ref{mainthm} and Corollary \ref{mainthmC}}
As a direct application of Theorem \ref{mainthm}, we will derive interesting results on the Erd\H{o}s-Falconer distance and the Box distance problem over finite fields, which will be introduced below. 
\subsubsection{On the Erd\H{o}s-Falconer distance}
We state the results on the Erd\H{o}s-Falconer distance problem determined by a set \( A \subseteq \mathbb{F}_q^d \), which can be derived from our Theorem \ref{mainthm}.

\begin{proposition}\label{ProK} Let $0< \alpha \le 1$. Suppose that $A\subseteq \mathbb F_q^d$ contains a subset $B$ of a $k$-coordinatable plane with $|B|\ge |A|^\alpha$. Then if $|A|\ge 2^{1/(\alpha+1)} q^{d/(1+\alpha)}$, we have $|\Delta(A)|\ge q/2$.
\end{proposition}
\begin{itemize}
\item Notice from the above proposition that for the sets $A\subseteq \mathbb F_q^d$ with $\frac{d-1}{d+1} <\alpha\le 1$, we improve the $(d+1)/2$ result. 
\item For $d=2$, the current best known threshold $4/3$ can be improved for the sets $A\subseteq \mathbb F_q^2$ with $1/2<\alpha \le 1$. In particular, for the prime field $\mathbb F_q$, the best known result $5/4$ can be improved for the sets $A\subseteq \mathbb F_q^2$ with $3/5<\alpha$.
\end{itemize}

For $j\in \mathbb F_q$, let $P_j=\{(x_1, x_2, \ldots, x_{d-1}, j)\in \mathbb F_q^d: x_i\in \mathbb F_q,\,i=1,2,\ldots, d-1\}$. Since $P_j$ is a $(d-1)$-coordinate plane and $\mathbb F_q^d=\bigcup_{j\in \mathbb F_q} P_j$, for any set $A\subseteq \mathbb F_q^d$, there exists a $P_j$ with $|P_j\cap A|\ge |A|/q$.
Hence, any set $A\subseteq \mathbb F_q^d$ contains a $(d-1)$-coordinate plane $P_j\cap A$ with cardinality at least $|A|/q$. Thus, the following corollary follows immediately from Theorem \ref{mainthm}.
\begin{corollary}\label{maincor} If $A\subseteq \mathbb F_q^d$ with $|A|\ge \sqrt{2} q^{\frac{d+1}{2}}$, then $|\Delta(A)|\ge q/2$.
\end{corollary}

The exponent $(d+1)/2$ in Corollary \ref{maincor} is sharp in odd dimensions and was obtained from Theorem \ref{mainthm}. Hence, Theorem \ref{mainthm} is generally sharp in odd dimensions.

\

\subsubsection{Result on the Box distance problem}
To study a variation of the Erd\H{o}s-Falconer distance problem, given \( E \subseteq \mathbb{F}_q^d, d\ge 1 \), Borges, Iosevich, and Ou \cite{BIO} defined  
\[
\Box(E) \coloneqq \{\|x - y\| + \|x - z\| : x, y, z \in E, \, y \neq z\},
\] 
and studied the problem of determining the smallest exponent $\beta>0$ such that 
if $|E|\ge C q^\beta$ for a sufficiently large constant $C>0$ independent of $q$, then $|\Box(E)|\sim q$. We will refer to this problem as the Box distance problem.
They proved the following result.

\begin{theorem} [Borges-Iosevich-Ou, \cite{BIO}, Theorem 2.2] \label{AlexThm}
If $E\subseteq \mathbb F_q^d$, $d\ge 1$, then the following statements hold:
\begin{itemize}
\item [(i)] If $d = 1$ and $|E| \ge Cq^{\frac{3}{4}}$, with $C$ sufficiently large, then $|\Box(E)|\sim q$.
\item [(ii)] If $d = 2$ and $|E| \ge Cq^{\frac{4}{3}}$, with $C$ sufficiently large, then $|\Box(E)|\sim q$.
\item [(iii)] If $d \ge 3$ and $|E| \ge C q^{\frac{d+1}{2}}$, with $C$ sufficiently large, then $|\Box(E)|\sim q$.
\end{itemize}
\end{theorem}

If we consider only the finite field $\mathbb F_q$ in which $2$ is a square, namely $\eta(2)=1$, we are able to improve the first part of Theorem \ref{AlexThm}.
More precisely, we have the following result.

\begin{theorem} \label{ThmK} Suppose that $2$ is a square in $\mathbb F_q$.
If $E\subseteq \mathbb F_q$ and $|E|\ge C q^{\frac{2}{3}}$ for some large constant $C$, then $|\Box(E)|\ge q/2$.
\end{theorem}

The remaining sections of this paper are organized to prove the results stated in the introduction:
Proposition \ref{P1.1}, Theorem \ref{mainthm},  Proposition \ref{ProK}, and Theorem \ref{ThmK}.

\section{Construction of sets for the Erd\H{o}s-Falconer distance conjecture}
We prove Proposition \ref{P1.1}, which shows that in odd dimensions $d$, no threshold result for the Erd\H{o}s-Falconer distance problem can exist below $(d+1)/2$, and in even dimensions $d$, below $d/2$. Our proof follows the argument that is implicitly contained in \cite{KR24}.

For the reader's convenience, we begin by reviewing the canonical forms of non-degenerate quadratic forms, as summarized in \cite{KR24}. 
Let \( P(x):=\|x\| \in \mathbb{F}_q[x_1, \ldots, x_d] \) be a non-degenerate quadratic form. Then we can write
\[
P(x) = x^\top \mathbf{I}_d ~x,
\]
where $\mathbf{I}_d$ denotes the $d\times d$ identity matrix.  
If we apply a linear change of variables by setting \( x = \mathbf{N}y \), where \( \mathbf{N}\) is an invertible \( d \times d \) matrix, and let $\mathbf{M}=\mathbf{N}^{\top}\mathbf{I}_d\mathbf{N}$, then the form transforms as follows:
\[
Q(y) =  y^\top \mathbf{M} y.
\]
In this case, we say that \( P(x) \) is equivalent to \( Q(y) \). Moreover, it is straightforward to observe that
\[
\Delta(A) = \Delta_{Q}(A'):=\{Q(y-y'): y, y'\in A'\},
\]
where \( A \subseteq \mathbb{F}_q^d \) and \( A' := \{N^{-1} x : x \in A\} \). Additionally, note that \( |A| = |A'| \), and the size of any set \( A \) is primarily considered as a hypothesis for distance-type problems. Therefore, without loss of generality, in the proof of Proposition  \ref{P1.1}, we may assume any non-degenerate quadratic form \( Q(x) \), equivalent to \( \|x\|\), to serve as a distance set.

The following non-degenerate quadratic form is a standard equivalent representation of $\|x\|$ (for a more general result, see Theorem 1 in \cite{AM16} or page 79 in \cite{Gr02}).

\begin{itemize}
\item[(i)] When \( d \) is even, the form \( \|x\| \) is equivalent to  
\begin{equation}\label{evenQ}  
Q(x)=\|x\|_Q := \sum_{i=1}^{d-1} (-1)^{i+1} x_i^2 - \varepsilon x_d^2 = x_1^2 - x_2^2 + \cdots + x_{d-1}^2 - \varepsilon x_d^2,  
\end{equation}  
where \( \varepsilon \in \mathbb{F}_q^* \) is chosen such that \( \eta((-1)^{\frac{d}{2}} \varepsilon) = 1 \).

\item[(ii)] When \( d \) is odd, the form \( \|x\| \) is equivalent to  
\begin{equation}\label{oddQ}  
Q(x)=\|x\|_Q := \sum_{i=1}^{d-1} (-1)^{i+1} x_i^2 + \varepsilon x_d^2 = x_1^2 - x_2^2 + \cdots + x_{d-2}^2 - x_{d-1}^2 + \varepsilon x_d^2,  
\end{equation}
where \( \varepsilon \in \mathbb{F}_q^* \) is chosen such that \( \eta((-1)^{\frac{d-1}{2}} \varepsilon) = 1 \).
\end{itemize}

\begin{remark} In both \eqref{evenQ} and \eqref{oddQ}, the parameter \( \varepsilon \) can be chosen to depend only on the dimension \( d \). Indeed, if $d\equiv 0 \pmod{4}$ or $d\equiv 1 \pmod{4}$, then \( \varepsilon \) can be chosen to be 1; otherwise, it can be taken to be \(-1\).
\end{remark}
We also need the following fact.

\begin{lemma}\label{lemCon} For each $0<\delta<1$, there exists a set $\Omega_\delta\subseteq \mathbb F_q$ such that
$|\Omega_\delta|\sim |\Omega_\delta -\Omega_\delta|\sim  q^{1-\delta}$.
\end{lemma}
\begin{proof}[Proof of Lemma \ref{lemCon}] 
Let \( q = p^{\ell} \) for some integer \( \ell \geq 1 \), where \( p \) is an odd prime. Then the finite field \( \mathbb{F}_q \) can be regarded as an \( \ell \)-dimensional vector space over \( \mathbb{F}_p \), with a basis given by \( \{1, \xi, \xi^2, \ldots, \xi^{\ell - 1}\} \), where \( \xi \) is an algebraic element over \( \mathbb{F}_p \) of degree \( \ell \).
For any \( 0 < \delta < 1 \), let $C_\delta \subseteq \mathbb F_p$ denote an arithmetic progression with $|C_\delta|\sim p^{1-\delta}$. Notice that $|C_\delta- C_\delta|\sim |C_\delta|$. Now for each $0<\delta<1$, we define a set 
$$\Omega_\delta:=\left\{c_0+c_1\xi+\dots+c_{\ell-1}\xi^{\ell-1}: c_0,\dots,c_{\ell-1}\in C_{\delta}\right\}\subseteq \mathbb{F}_q.$$
It is clear that $|\Omega_\delta|=|C_\delta|^\ell\sim  p^{\ell(1-\delta)}=q^{1-\delta}$. Furthermore, since the difference set of $\Omega_\delta$ is defined as
$$\Omega_{\delta}-\Omega_{\delta}:=\left\{(c_0-c_0')+(c_1-c_1')\xi+\dots+(c_{\ell-1}-c_{\ell-1}')\xi^{\ell-1}: c_0,c_0',\dots,c_{\ell-1},c_{\ell-1}'\in C_{\delta}\right\},$$ 
we see that $|\Omega_{\delta}-\Omega_{\delta}|=|C_{\delta}-C_{\delta}|^{\ell}\sim |C_{\delta}|^{\ell}\sim q^{1-\delta}$, which completes the proof.
\end{proof}

\subsection{Proof of Proposition~\ref{P1.1}, part~(i)}
Suppose $d\ge 3$ is odd. Then, by \eqref{oddQ}, instead of the usual distance set $\Delta(A)$ we may consider the following distance set
$$ \Delta_Q(A):=\{\|x-y\|_Q: x, y\in A\},$$
where $\|x-y\|_Q=(x_1-y_1)^2 -(x_2-y_2)^2 + \cdots+ (x_{d-2}-y_{d-2})^2- (x_{d-1}-y_{d-1})^2 +\varepsilon (x_d-y_d)^2$.
For any $0<\delta<1$, define 
$A=H\times \Omega_\delta\subseteq \mathbb F_q^{d-1}\times \mathbb F_q$,
where $\Omega_\delta$ is the set given as in Lemma \ref{lemCon} and the set $H\subseteq \mathbb F_q^{d-1}$ is defined by
$$ H:=\left\{(t_1, t_1, \ldots, t_i, t_i, \ldots, t_{(d-1)/2}, t_{(d-1)/2}) \in \mathbb F_q^{d-1}: t_i\in \mathbb F_q,  1\le i\le (d-1)/2\right\}.$$
Then it is not hard to see that
$$|A|=|H||\Omega_\delta|\sim q^{(d-1)/2} q^{1-\delta} =q^{\frac{d+1}{2}-\delta},$$
and from Lemma \ref{lemCon}, we have 
$$|\Delta_Q(A)|=|\{\varepsilon (a-b)^2: a, b\in \Omega_\delta\}|\sim |\Omega_\delta-\Omega_\delta|\sim q^{1-\delta}.$$
This concludes the proof of part (i).

\subsection{Proof of Proposition~\ref{P1.1}, part~(ii)}
Let the dimension $d\ge 2$ be even.  By \eqref{evenQ}, we may consider the distance set 
$$ \Delta_Q(A):=\{\|x-y\|_Q: x, y\in A\},$$
where $\|x-y\|_Q=(x_1-y_1)^2 -(x_2-y_2)^2 + \cdots+ (x_{d-2}-y_{d-1})^2 -\varepsilon (x_d-y_d)^2$. Now for any $0<\delta<1$, we set
$$ A=\Lambda \times \Omega_\delta\times \{0\} \subseteq \mathbb F_q^{d-2}\times \mathbb F_q \times \mathbb F_q,$$
where the set $\Omega_\delta\subseteq \mathbb F_q$ is defined as in Lemma \ref{lemCon} and the set $\Lambda \subseteq \mathbb F_q^{d-2}$ is defined as
$$ \Lambda=\left\{(t_1, t_1, \ldots, t_i, t_i, \ldots, t_{(d-2)/2}, t_{(d-2)/2}) \in \mathbb F_q^{d-2}:  t_i\in \mathbb F_q,  1\le i\le (d-2)/2\right\}.$$
Then one can easily check that $|A|=|\Lambda||\Omega_\delta|\sim q^{\frac{d-2}{2}} q^{1-\delta}= q^{\frac{d}{2}-\delta}$ and 
$$|\Delta_Q(A)| =\{(a-b)^2: a, b\in \Omega_\delta\}|\sim |\Omega_\delta -\Omega_\delta|\sim |\Omega_\delta|\sim q^{1-\delta},$$ as required.

\section{Preliminaries}

In this section, we record the results that will be used in the proofs of our main results, Theorem~\ref{mainthm} and Theorem~\ref{ThmK}. To this end, we begin by reviewing background material on discrete Fourier analysis and Gauss sums, which serve as key tools in this paper and can be found in \cite{IR07, IK04, LN97}.

\subsection{Foundational Tools}
If \( f: \mathbb{F}_q^d \to \mathbb{C} \) is a function, then its Fourier transform $\widehat{f}:\mathbb{F}_q^d\to \mathbb{C}$ is defined by  
\[
\widehat{f}(m) \coloneqq q^{-d} \sum_{x \in \mathbb{F}_q^d} \chi(-m \cdot x) f(x),
\]  
where \( m \cdot x \) denotes the standard dot product in $\mathbb{F}_q^d$, and $\chi:\mathbb{F}_q\to S^1\subseteq \mathbb{C}$ is a nontrivial additive character.

The most important property of \( \chi \) is its orthogonality relation:
\[
\sum_{x \in \mathbb{F}_q^d} \chi(m \cdot x) =
\begin{cases}
0, & \text{if } m \ne (0, \ldots, 0), \\
q^d, & \text{if } m = (0, \ldots, 0).
\end{cases}
\]
By the orthogonality of \( \chi \) and the definition of the Fourier transform, the following Fourier inversion theorem can be easily derived:
\[
f(x) = \sum_{m \in \mathbb{F}_q^d} \chi(m \cdot x)\, \widehat{f}(m).
\]
By similar reasoning, one can deduce the following important identity, known as the Plancherel theorem:
\[
\sum_{m \in \mathbb{F}_q^d} |\widehat{f}(m)|^2 = q^{-d} \sum_{x \in \mathbb{F}_q^d} |f(x)|^2.
\]
In particular, when the function \( f \) is taken to be the indicator function of a set \( E \subseteq \mathbb{F}_q^d\), the following identity follows immediately from the Plancherel theorem:
\[
\sum_{m \in \mathbb{F}_q^d} |\widehat{E}(m)|^2 = q^{-d} |E|.
\]
Here, and throughout this paper, we identify a set \( E \) with its indicator function \( 1_E \).

For each element \( a \in \mathbb{F}_q^* \), the Gauss sum \( \mathcal{G}_a \) is defined as
\[
\mathcal{G}_a := \sum_{s \in \mathbb{F}_q^*} \eta(s)\, \chi(as),
\]
where \( \eta \) is the quadratic character on \( \mathbb{F}_q^* \) and \( \chi \) is a nontrivial additive character of $\mathbb{F}_q$.
Alternatively, \( \mathcal{G}_a \) can be also expressed as
\[
\mathcal{G}_a = \sum_{s \in \mathbb{F}_q} \chi(as^2) = \eta(a)\, \mathcal{G}_1,
\]
for all \( a \in \mathbb{F}_q^* \).
By completing the square and performing an appropriate change of variables, one obtains the following identity:
\begin{equation}\label{ComSqu}  
\sum_{s \in \mathbb{F}_q} \chi(as^2 + bs) = \eta(a)\ \chi\left(\frac{b^2}{-4a}\right)\mathcal{G}_1.
\end{equation}
We can easily derive the following result for the Gauss sum:
\begin{equation}\label{Corm}
\mathcal{G}_1^2 = \eta(-1) q.
\end{equation}
Indeed, since the modulus of the Gauss sum \( \mathcal{G}_a \) is exactly \( \sqrt{q} \), it is evident that \( \mathcal{G}_a \overline{\mathcal{G}_a} = |\mathcal{G}_a|^2 = q \) for any \( a \in \mathbb{F}_q^* \). Moreover, by performing a change of variables, we observe that
\[
\overline{\mathcal{G}_a} = \sum_{s \in \mathbb{F}_q^*} \eta(s) \chi(-as) = \eta(-1) \mathcal{G}_a.
\]
Thus, the desired equality in \eqref{Corm} is obtained.

\subsection{Homogeneous varieties in $\mathbb F_q^d \times \mathbb F_q^d$} 
Let $d\ge 2$ be an integer.
For $X=(x, y)\in \mathbb F_q^{d} \times \mathbb F_q^{d}=\mathbb F_q^{2d}$, we define 
\begin{equation}\label{starDef} \|X\|_*:= \|x\|-\|y\|.\end{equation}
\begin{definition} \label{DefV0}  Let $d\ge 2$ be an integer. A homogeneous variety, denoted by $V_0$, is a variety defined by
$$ V_0:=\{X\in \mathbb F_q^{2d}: \|X\|_*=0\}.$$
\end{definition}

The Fourier transform of the homogeneous variety $V_0$ in $\mathbb{F}_q^{2d}$ is explicitly given as follows. For the convenience of readers, the proof given in \cite{CKP21} is included below.

\begin{lemma} [{\cite[Lemma 3.2]{CKP21}}]
\label{defVFT} 
Let $d\ge 2$ be an integer. If $M \in \mathbb{F}_q^{2d}$, then we have:
\[
\widehat{V_0}(M) := q^{-2d} \sum_{X \in V_0} \chi(-M \cdot X) =
\begin{cases}
q^{-1} \delta_0(M) + q^{-d-1}(q - 1), & \text{if } \|M\|_* = 0, \\
- q^{-d-1}, & \text{if } \|M\|_* \ne 0,
\end{cases}
\]
where $\delta_0(M)=1$ if $M=(0,\ldots, 0)$ and $0$ otherwise.
\end{lemma}
\begin{proof}[Proof of Lemma \ref{defVFT}]
It follows from the orthogonality of $\chi$ that  
\[
\widehat{V_0}(M) = q^{-2d} \sum_{X \in V_0} \chi(-M \cdot X) = q^{-1} \delta_0(M) + q^{-2d-1} \sum_{X \in \mathbb{F}_q^{2d}} \sum_{s \ne 0} \chi(s\|X\|_* - M \cdot X).
\]

Applying the complete square formula~\eqref{ComSqu}, we obtain  
\[
\widehat{V_0}(M) = q^{-1} \delta_0(M) + q^{-2d-1}  \mathcal{G}_1 ^{2d} \sum_{s \ne 0} \eta^{d}(s)\eta^{d}(-s) \chi\left( \frac{\|M\|_*}{-4s} \right).
\]

Since $\eta^{2d} = 1$, a change of variables yields  
\[
\widehat{V_0}(M) = q^{-1} \delta_0(M) + q^{-2d-1} \mathcal{G}_1^{2d} \eta^d(-1)\sum_{r \ne 0} \chi(r\|M\|_*).
\]

Notice from \eqref{Corm} that $\mathcal{G}_1^{2d} \eta^d(-1)=q^d$.
By the orthogonality of $\chi$, we finish the proof.
\end{proof} 
\section{Key Lemmas}
In this section, we present important facts that will be directly applied in proving the main results of this paper.
\subsection{A lower bound for the size of $\Delta(A, B)$}
We derive a formula for a lower bound of the distance set \( \Delta(A, B) \), expressed from the perspective of the restriction estimate for the sphere $S_t^{d-1}:=\{x\in \mathbb F_q^d: \|x\|=t\}$ for $t\in \mathbb F_q$.
Although a similar formula has been previously derived, our proof will yield a simpler and more refined result by applying Lemma \ref{defVFT}  instead of computing the Fourier transform on spheres. More precisely, we obtain the following formula.

\begin{lemma} \label{DistFormula} Let $A$, $B\subseteq \mathbb F_q^d$. Then we have
$$ |\Delta(A, B)| \ge \frac{|A|^2|B|^2}{ q^{-1}|A|^2 |B|^2 + q^{2d} |A|  \Bigg(\max\limits_{t\in \mathbb F_q} \sum\limits_{m\in S_t^{d-1}} |\widehat{B}(m)|^2\Bigg)}.$$
\end{lemma}

\begin{proof}[Proof of Lemma \ref{DistFormula}] 
We begin by using a well-known counting method. For \( A, B \subseteq \mathbb{F}_q^d \) and \( t \in \mathbb{F}_q \), we define the counting function \( \nu(t) \) as follows: $$\nu(t)\coloneqq |\{(x,y)\in A\times B: \lVert x-y\rVert=t\}|.$$
As is standard, the Cauchy--Schwarz inequality implies the following lower bound:
\[
|\Delta(A, B)| \ge \frac{|A|^2 |B|^2}{\sum\limits_{t \in \mathbb{F}_q} \nu^2(t)}.
\]
Therefore, to complete the proof,  it suffices to prove that 
\begin{equation}\label{FinalAim}
\sum_{t \in \mathbb{F}_q} \nu^2(t) \le q^{-1}|A|^2 |B|^2 + q^{2d} |A|  \Bigg(\max_{t\in \mathbb F_q} \sum\limits_{m\in S_t^{d-1}} |\widehat{B}(m)|^2\Bigg).
\end{equation}
By the definition of $\nu(t)$, we have
$$\sum\limits_{t \in \mathbb{F}_q} \nu^2(t)=\sum_{\substack{x, z\in A, y, w\in B\\ \|x-y\|-\|z-w\|=0}}  1.$$
Let $X=(x,z)\in A\times A$ and $Y=(y,w)\in B\times B$. Then, by the definition of $\|\cdot \|_*$ in \eqref{starDef}, we have
$$\sum_{t \in \mathbb{F}_q} \nu^2(t)= \sum_{\substack{X\in A\times A, Y\in B\times B\\ \|X-Y\|_*=0}} 1=\sum_{\substack{X\in A\times A\\ Y\in B\times B}} V_0(X-Y),$$ 
where $V_0$ is the variety in $\mathbb F_q^{2d}$, defined in Definition \ref{DefV0}.
Applying the Fourier inversion theorem to the function $V_0(X-Y)$, it follows from the definition of the Fourier transform that 
$$\sum_{t \in \mathbb{F}_q} \nu^2(t)=q^{4d} \sum_{M\in \mathbb F_q^{2d}}   \widehat{V_0}(M)~ \overline{\widehat{A\times A}}(M) ~ \widehat{B\times B}(M).$$
After substituting the explicit expression for $\widehat{V_0}(M)$ given in Lemma \ref{defVFT} and simplifying via direct computation, we obtain the following equations:
\begin{align*}
\sum_{t \in \mathbb{F}_q} \nu^2(t)&=q^{4d-1} \overline{\widehat{A\times A}}(0, \ldots, 0)  ~  \widehat{B\times B}(0, \ldots, 0) + q^{3d-1}(q-1)\sum_{\substack{M\in \mathbb F_q^{2d}\\ \|M\|_*=0}}\overline{\widehat{A\times A}}(M)  ~  \widehat{B\times B}(M)\\
& -q^{3d-1} \sum_{\substack{M\in \mathbb F_q^{2d}\\ \|M\|_*\ne 0}}\overline{\widehat{A\times A}}(M)  ~  \widehat{B\times B}(M) \\
&=  \frac{|A|^2|B|^2}{q} + q^{3d} \sum_{\substack{M\in \mathbb F_q^{2d}\\ \|M\|_*=0}} \overline{\widehat{A\times A}}(M)  \widehat{B\times B}(M)-q^{3d-1} \sum_{M\in \mathbb F_q^{2d}}\overline{\widehat{A\times A}}(M)  ~  \widehat{B\times B}(M).\\
\end{align*}

For $M=(m, m')\in \mathbb F_q^d \times \mathbb F_q^d, $  notice that  $\overline{\widehat{A\times A}}(M) \widehat{B\times B}(M)= \overline{\widehat{A}}(m) \overline{\widehat{A}}(m') \widehat{B}(m) \widehat{B}(m') $  and
$$\sum\limits_{\substack{M\in \mathbb F_q^{2d}\\ \|M\|_*=0}} = \sum\limits_{\substack{m, m'\in \mathbb F_q^d\\ \|m\|=\|m'\|}} =\sum\limits_{t\in \mathbb F_q} \sum\limits_{\substack{m\in \mathbb F_q^d\\ \|m\|=t}} \sum\limits_{\substack{m'\in \mathbb F_q^d\\ \|m'\|=t}}.$$ 

Then we see that
\begin{align*}\sum_{t \in \mathbb{F}_q} \nu^2(t)&\le \frac{|A|^2|B|^2}{q} + q^{3d}\sum_{t\in \mathbb F_q} \left(\sum_{m\in S_t^{d-1}} \overline{\widehat{A}}(m)\widehat{B}(m)\right)^2 -q^{3d-1} \left( \sum_{m\in \mathbb F_q^d} \overline{\widehat{A}}(m)\widehat{B}(m)\right)^2\\
&\le \frac{|A|^2|B|^2}{q} + q^{3d}\sum_{t\in \mathbb F_q} \left(\sum_{m\in S_t^{d-1}} \overline{\widehat{A}}(m)\widehat{B}(m)\right)^2,\end{align*}
where the last inequality follows since the third term above is a non-negative real number.
By the Cauchy-Schwarz inequality, we see that
\begin{align*}
\sum_{t\in \mathbb F_q} \nu^2(t) &\le \frac{|A|^2|B|^2}{q} + q^{3d} \sum_{t\in \mathbb F_q} \left(\sum_{m\in S_t^{d-1}} |\widehat{A}(m)|^2\right) \left(\sum_{m\in S_t^{d-1}} |\widehat{B}(m)|^2\right)\\
& \le \frac{|A|^2|B|^2}{q} + q^{3d} \left(\max_{t\in \mathbb F_q} \sum_{m\in S_t^{d-1}} |\widehat{B}(m)|^2\right) \sum_{m\in \mathbb F_q^d} |\widehat{A}(m)|^2.
\end{align*}
Since $\sum\limits_{m\in \mathbb F_q^d} |\widehat{A}(m)|^2 = q^{-d}|A|$, inequality \eqref{FinalAim} follows, completing the proof.
\end{proof}

\subsection{$L^2$ restriction estimates for spheres}
We investigate \( L^2 \) restriction estimates for spheres in \( \mathbb{F}_q^d \), with test functions given by indicator functions of \( k \)-coordinatable sets. To this end, we need an explicit cardinality of the sphere $S_t^{d-1}$ in $\mathbb F_q^d$. Here we recall that for $t\in \mathbb F_q$ we define a sphere with radius $t$ as follows:
$$S_t^{d-1}:=\{x\in \mathbb F_q^d: \|x\|=t\}.$$

\begin{lemma}[{\cite[Theorems 6.26 and 6.27]{LN97}}] \label{SphereSize}
For $t\in \mathbb F_q$, we define $\omega(t)=-1$ for $t\ne 0$, and $\omega(0)=q-1$. 
\begin{itemize} \item[(i)] If $d$ is even, then  
$$\big|S_t^{d-1}\big|=q^{d-1}+ \omega(t) q^{\frac{d-2}{2}} \eta\left((-1)^{\frac{d}{2}}\right).$$
\item[(ii)] If $d$ is odd, then
$$\big|S_t^{d-1}\big|=q^{d-1}+ q^{\frac{d-1}{2}} \eta\left(t(-1)^{\frac{d-1}{2}} \right).$$ 
\end{itemize}
\end{lemma}

It is clear from Lemma \ref{SphereSize} that
\begin{equation}\label{SizeS} \max_{t\in \mathbb F_q} \big|S_t^{d-1}\big| \le 2 q^{d-1}.
\end{equation}



It follows from Lemma~\ref{DistFormula} that finding a lower bound for the distance set $\Delta(A, B)$ is related to the problem of obtaining an upper bound for the $L^2$ restriction estimate on the sphere, namely
\[  \mathcal{R}_t(B):=
\sum_{m \in S_t^{d-1}} |\widehat{B}(m)|^2.
\]
Since $S_t^{d-1} \subseteq \mathbb F_q^d,$  it obviously follows from the Plancherel theorem that 
$$ \mathcal{R}_t(B) \le q^{-d}|B|.$$
Although various efforts have been made to improve upon this seemingly obvious result, improvements for a general set \( B \)  have only been achieved in two dimensions. More precisely,  the authors in \cite{CEHIK10} showed that for $d=2$ and $t\ne 0$, 
$$ \mathcal{R}_t(B)\lesssim  q^{-3} |B|^{\frac{3}{2}}.$$
Our result below shows that this estimate can be improved in the case when  the set $B\subseteq \mathbb F_q^d$ is a $k$-coordinatable set. 
\begin{proposition} \label{ProRes} For any $k$-coordinatable set $B \subseteq \mathbb F_q^d$ with $1\le k\le d-1$, we have
$$\max_{t\in \mathbb F_q} \sum_{m\in S_t^{d-1}} |\widehat{B}(m)|^2 \le 2 q^{-d-1} |B|.$$ 
\end{proposition}
\begin{proof}[Proof of Proposition \ref{ProRes}] Notice that the sphere $S_t^{d-1}$  is invariant under a rotation
and  the absolute value of the Fourier transform on a set $B\subseteq \mathbb F_q^d$ is the same as that on its translation. Hence, without loss of generality, we may assume that the $k$-coorinatable set $B\subseteq \mathbb F_q^d$ is written as
$$ B=B_k \times \{\textbf{0}\},$$
where $B_k$ is a subset of $\mathbb F_q^k$ and  $\mathbf{0}$ denotes the zero vector  in $\mathbb F_q^{d-k}.$

For $B_k\subseteq \mathbb F_q^k$ and the zero vector $\mathbf{0}$ in $\mathbb F_q^{d-k}$, we observe that for $m=(m_1, \ldots, m_k, \mathbf{0}) \in \mathbb F_q^k \times \mathbb F_q^{d-k}$, we have
$$\widehat{B}(m)=\widehat{B_k\times \{\mathbf{0}\}}(m) = \widehat{B_k}(m_1, \ldots, m_k) \widehat{\{\mathbf{0}\}} (\mathbf{0})=\widehat{B_k}(m_1, \ldots, m_k) q^{-d+k}.$$
Hence, for each $t\in \mathbb F_q$, we can write
\begin{align*}
\sum_{m\in S_t^{d-1}} |\widehat{B}(m)|^2 &= \sum_{m_1, m_2,\ldots, m_k\in \mathbb F_q} |\widehat{B_k}(m_1, \ldots, m_k)|^2  \sum_{\substack{m_{k+1}, \ldots, m_d\in \mathbb F_q\\ m_{k+1}^2+ \cdots+ m_d^2=t-m_1^2-\cdots- m_k^2}} |\widehat{\{\mathbf{0}\}} (\mathbf{0})|^2\\
& =\sum_{m_1, m_2,\ldots, m_k\in \mathbb F_q}q^{-2d+2k} |\widehat{B_k}(m_1, \ldots, m_k)|^2 |S_{t-m_1^2-\cdots-m_k^2}^{d-k-1}|.  
\end{align*}
Now applying \eqref{SizeS}  and the Plancherel theorem,  we obtain that
\begin{align*}
\sum_{m\in S_t^{d-1}} |\widehat{B}(m)|^2&\le  \sum_{m_1, m_2,\ldots, m_k\in \mathbb F_q} 2 q^{d-k-1} q^{-2d+2k}|\widehat{B_k}(m_1, \ldots, m_k)|^2  \\ &=2q^{-k}q^{d-k-1} q^{-2d+2k}|B_k|  =2q^{-d-1} |B_k|.
\end{align*}
Since $|B|=|B_k|$, this completes the proof.
\end{proof}

\subsection{1-coordinatable plane in two dimensions}

\begin{lemma}\label{1Rotation}  Let $(a, b)\in \mathbb F_q^2$ and let $\lambda\in \mathbb F_q^*.$ If $\eta(1+\lambda^2)=1,$ then  the line $L_\lambda(a,b)$ is  1-coordinatable  in $\mathbb F_q^2.$ \end{lemma}

\begin{proof}[Proof of Lemma \ref{1Rotation}] 
Recall from Definition \ref{DefTL} that  \(L_\lambda(a, b)\) denotes the line in the two-dimensional plane with slope \(\lambda\), translated by the vector \((a, b)\). By the definition of a 1-coordinatable plane, without loss of generality,  we may assume that $(a, b)=(0,0)$, and  it suffices to prove that there exists a rotation $R_\lambda$ that maps the line through the origin with slope \(\lambda\) onto  the \(x\)-axis.  

We now proeed to prove this. Since $\eta(1+\lambda^2)=1,$  there exists
$\sqrt{1+\lambda^2}\in \mathbb F_q^*$. Define $$R_\lambda \coloneqq
\begin{pmatrix}
\frac{1}{\sqrt{1+\lambda^2}} & \frac{\lambda}{\sqrt{1+\lambda^2}} \\
\frac{-\lambda}{\sqrt{1+\lambda^2}} & \frac{1}{\sqrt{1+\lambda^2}}
\end{pmatrix}.$$
It is clear that $ R_\lambda \in SO_2(\mathbb{F}_q)$, where
$$
SO_2(\mathbb{F}_q) \coloneqq \big\{ A \in GL_2(\mathbb{F}_q) |\ A^\top A = I,\ \det(A) = 1 \big\}.
$$
Moreover, if $(x_1, x_2)\in L_\lambda$, then $x_2=\lambda x_1$, and so 
$$
R_\lambda \begin{pmatrix} x_1 \\  x_2  \end{pmatrix}=R_\lambda \begin{pmatrix} x_1 \\ \lambda x_1  \end{pmatrix} 
=\begin{pmatrix} x_1 \sqrt{1+\lambda^2} \\ 0\end{pmatrix}.$$
Hence, under the rotation transformation \( R_\lambda \), the line \( L_\lambda \) is mapped to the \( x \)-axis, as required.
\end{proof}

\section{Proofs of results}
In this section, we provide proofs of our results: Theorem \ref{mainthm}, Proposition \ref{ProK}, and Theorem \ref{ThmK}.

\subsection{Proof of Theorem \ref{mainthm}}
Since $\Delta(A,B)=\Delta(B, A)$ for $A, B\subseteq \mathbb F_q^d$, we may assume that $B\subseteq \mathbb F_q^d$ is a subset of a $k$-coordinatable plane. Hence,
to complete the proof of Theorem \ref{mainthm},  it will be enough to prove the following theorem:

\begin{theorem} If $A\subseteq \mathbb F_q^d$ is an arbitrary subset and $B\subseteq \mathbb F_q^d$ is a subset of a $k$-coordinatable plane,  then  
$$ |\Delta(A, B)|\ge \frac{1}{2} \min\left\{ q, ~ \frac{|A||B|}{2 q^{d-1}}\right\}.$$
\end{theorem}
To prove this theorem,  we combine Lemma \ref{DistFormula} and Proposition \ref{ProRes}. Then we obtain the desired result:
$$|\Delta(A, B)| \ge \frac{|A|^2|B|^2}{ q^{-1}|A|^2 |B|^2 + q^{2d} |A| 2q^{-d-1} |B| }\ge \frac{1}{2} \min\left\{ q, ~ \frac{|A||B|}{2q^{d-1}}\right\}.$$

\subsection{Proof of Proposition \ref{ProK}}
For the reader’s convenience, we restate Proposition \ref{ProK} in an alternative form below and provide its proof.
\begin{proposition}\label{ProK1}
Assume \( 0< \alpha \le 1 \), and let \( A\subseteq \mathbb{F}_q^d \) be a set containing a subset \( B \) that lies in a \( k \)-coordinatable plane and satisfies \( |B|\ge |A|^\alpha \). If it further holds that \( |A|\ge 2^{1/(\alpha+1)} q^{d/(1+\alpha)} \), then we conclude that \( |\Delta(A)|\ge q/2 \).
\end{proposition}
\begin{proof}[Proof of Proposition \ref{ProK}] Since $A \supseteq B,$  it is clear that
$$ |\Delta(A)|=|\Delta(A, A)| \ge |\Delta(A, B)|.$$
It follows that $|A||B|\ge |A|^{1+\alpha} \ge 2 q^d,$ since $|B|\ge |A|^\alpha$  and \( |A|\ge 2^{1/(\alpha+1)} q^{d/(1+\alpha)} \).
Since $B$ is a $k$-coordinatable plane, we finish the proof by using Theorem  \ref{mainthm}. 
\end{proof}

\subsection{Proof of Theorem \ref{ThmK}}
Suppose that $2\in \mathbb F_q$ is a square and  $E\subseteq \mathbb F_q$ with $|E|\ge Cq^{2/3}$ for a sufficiently large constant $C.$ 
Then, by Lemma \ref{1Rotation}, the line $B:=\{(x,x): x\in E\}$ is a subset of a 1-coordinatable plane in $\mathbb F_q^2.$  
Consider  two subsets \( E_1 \) and \( E_2 \) of a set \( E \) satisfying the following conditions:
\begin{itemize}
  \item \( E_1 \cap E_2 = \emptyset \),
  \item \( E_1 \cup E_2 = E \),
  \item \( 0\le |E_1| - |E_2|  \leq 1 \).
  \end{itemize}
Now define $A=E_1\times E_2\subseteq \mathbb F_q^2.$ Then $|A|=|E_1||E_2|\ge |E|(|E|-1)/4.$  Notice that $\|x-y\|+\|x-z\|= \|(x,x)-(y,z)\|$ for $x,y,z\in E.$ Then, by the definitions of $E_1$ and $E_2$,  we see that
\begin{align*}
\Box(E) &= \{\|x - y\| + \|x - z\| : x, y, z \in E, \, y \neq z\}\\
&\supseteq \{\|x - y\| + \|x - z\| : x\in E, y\in E_1, z\in E_2\}\\
&=\{\|(x,x)-(y,z)\| : x\in E, (y, z)\in E_1\times E_2=A\}\\
&= \Delta(A, B).
\end{align*}
In other words,  we get
\begin{equation}\label{Boxsize} |\Box(E)|\ge |\Delta(A, B)|.\end{equation}
Since $$|A||B|=|E_1||E_2| |E|\ge \frac{|E|^2(|E|-1)}{4}\sim |E|^3,$$ it follows that if $|E|\ge Cq^{2/3}$,  then $|A||B|\ge 2 q^2.$  Therefore,  Theorem \ref{mainthm} implies that $ |\Delta(A,B)|\ge q/2.$
By \eqref{Boxsize},  we obtain the desired result:  $|\Box(E)|\ge q/2.$
\bibliographystyle{amsplain}

\end{document}